\documentclass{amsart}
\usepackage{amsmath, amsfonts,graphicx,subfigure,cite,comment}\usepackage{amsmath,amsfonts,amsthm,amssymb,graphics,color,datetime}
\usepackage{tikz} 

\usepackage{hyperref}
\theoremstyle{theorem}
\newtheorem{theorem}{Theorem}
\newtheorem{cor}{Corollary}
\newtheorem{definition}{Definition}
\newtheorem{lemma}{Lemma}
\newtheorem{prop}{Proposition}
\newtheorem{conj}{Conjecture}
\newtheorem{question}{Question}
 
\theoremstyle{definition}
\newtheorem*{remark}{Remark}
\newcommand{\GL}{\operatorname{GL}}
\newcommand{\Ord}{\operatorname{Ord}}

\newcommand{\pa}{\partial}
\newcommand{\eps}{\varepsilon}

\newcommand{\bp}{\textbf{p}} 
\newcommand{\bx}{\textbf{x}} 
\newcommand{\bv}{\textbf{v}}
\newcommand{\by}{\textbf{y}}
\newcommand{\bu}{\textbf{u}}
\newcommand{\bw}{\textbf{w}}
\newcommand{\bL}{\textbf{L}}
\newcommand{\bM}{\textbf{M}}
\newcommand{\bgamma}{\boldsymbol\gamma}

\newcommand{\bq}{\textbf{q}}

\newcommand{\bz}{{\bf z}}

\newcommand{\cP}{{\mathcal{P}}}

\newcommand{\N}{\mathbb{N}}
\newcommand{\R}{\mathbb{R}}

\newcommand{\C}{\mathbb{C}}
\newcommand{\Z}{\mathbb{Z}}

\newcommand{\cL}{\mathcal{L}}

\begin{document}

 \title[Crystals, polytopes, and eigenfunctions]{Crystallographic groups, strictly tessellating polytopes, and analytic eigenfunctions} 

%\markright{Crystals, polytopes, and eigenfunctions}
\author{}
\author[Rowlett, Blom, Nordell, Thim, \& Vahnberg]{Julie Rowlett, Max Blom, Henrik Nordell, \\ Oliver Thim, and Jack Vahnberg}

\maketitle

\begin{abstract} 
The mathematics of crystalline structures connects analysis, geometry, algebra, and number theory.  The planar crystallographic groups were classified in the late 19th century.  One hundred years later, B\'erard proved that the fundamental domains of all such groups satisfy a very special analytic property:  the Dirichlet eigenfunctions for the Laplace eigenvalue equation are all trigonometric functions.  In 2008, McCartin proved that in two dimensions, this special analytic property has both an equivalent algebraic formulation, as well as an equivalent geometric formulation.  Here we generalize the results of B\'erard and McCartin to all dimensions.  We prove that the following are equivalent:  the first Dirichlet eigenfunction for the Laplace eigenvalue equation on a polytope is real analytic, the polytope strictly tessellates space, and the polytope is the fundamental domain of a crystallographic Coxeter group.  Moreover, we prove that under any of these equivalent conditions, all of the eigenfunctions are trigonometric functions.  In conclusion, we connect these topics to the Fuglede and Goldbach conjectures and give a purely geometric formulation of Goldbach's conjecture. 
\end{abstract} 

\section{Introduction}
In \em The Grammar of Ornament, \em published in 1856, Owen Jones wrote \cite{jones}: 
\begin{quote} 
 Whenever any style of ornament commands universal admiration, it will always be found to be in accordance with the laws which regulate the distribution of forms in nature.  
\end{quote} 
In the case of crystals, the laws which regulate their shape are dictated by the crystallographic groups.  

\subsection{Crystallographic groups} 
A crystal or crystalline solid is a solid material whose constituents, such as atoms, molecules, or ions, are arranged in a highly ordered microscopic structure; for a two dimensional example see Figure \ref{fig:graph}.  Mathematically we can identify the locations of the constituents with the points of a lattice, also known as a \em crystal lattice.  \em  Every two dimensional crystal has a symmetry group which is a plane crystallographic group.  Such a group consists of isometries of the plane. There are three basic types of isometries:  translations, rotations and reflections.  These form a group under composition.  The patterns in Figure \ref{fig1} have symmetry groups which are \em plane crystallographic groups.  \em  These are subgroups of the group of isometries of the plane which are topologically discrete and contain two linearly independent translations.  Equivalently, a plane crystallographic group\footnote{These are also known as wallpaper groups.} is a co-compact subgroup of the group of isometries of the plane. A subgroup in this context is called co-compact if the quotient space $\R^2/\Gamma$ by the subgroup, $\Gamma$, is compact.  The classification of these groups, up to equivalence, was achieved at the end of the 19th century by E. S. Fedorov \cite{fed} and A. Schoenflies \cite{schoen}.  Two groups are equivalent if they are equal up to a translation. In two dimensions, up to this notion of equivalence, there are seventeen crystallographic groups. 

\begin{figure} \centering \includegraphics[width=8cm]{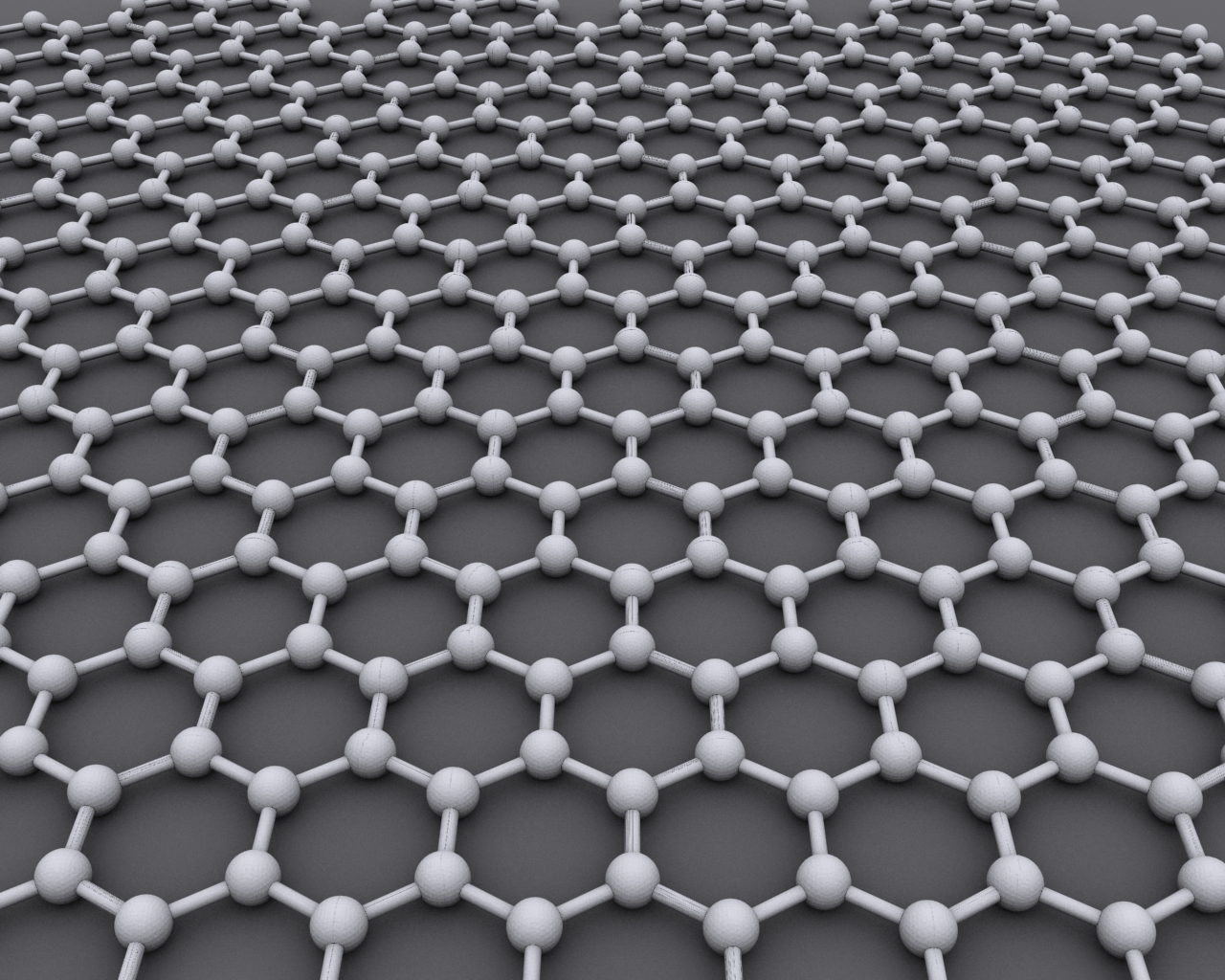} \caption{Graphene is an allotrope of carbon in the form of a two-dimensional, atomic-scale hexagonal lattice such that each point in the lattice corresponds to an atom.  This image is licensed under the Creative Commons Attribution-Share Alike 3.0 Unported license at \url{https://commons.wikimedia.org/wiki/File:Graphen.jpg}.} \label{fig:graph}  \end{figure}

One can also consider crystals in three dimensions, and mathematically we may generalize all of these notions to $\R^n$.  An $n$-dimensional crystallographic group is a discrete group of isometries of $\R^n$ which is co-compact.  Fedorov \cite{fed} and Schoenflies \cite{schoen} proved that there are, up to equivalence, 219 crystallographic groups in $\R^3$.  In 1910, Bieberbach proved that for any $n$, there are only finitely many $n$-dimensional crystallographic groups up to equivalence \cite{bieber1, bieber2}, thereby solving Hilbert's 18th problem.  

%%%%%%%%%%%%%%%%%%%%%%%%%%%%
\subsection{Strictly tessellating polytopes and our main result}  
Crystal lattices create a perfectly regular pattern, and the fundamental domain of the associated full-rank lattice tessellates space by translation.  Another way to create a perfectly regular pattern is by `strict tessellation.'  This is a notion specific to polytopes. 

\begin{definition} \label{def:polytope} 
The set of all one dimensional polytopes is the set of all bounded open intervals, 
$$\wp_1 := \{ (a,b) : -\infty < a < b < \infty \}.$$
Inductively we define the set of polytopes in $\R^n$ for $n\geq 2$ to be the set of connected bounded domains in $\R^n$ such that $\Omega \in \wp_n$ if and only if 
$$\pa \Omega = \bigcup_{j=1} ^m \overline{P_j}, \quad P_j \cong Q_j \in \wp_{n-1}.$$
Above, the boundary of $\Omega$ consists of the closures of $n-1$ dimensional polytopes, $P_j$.  Each $P_j$ is contained in an $n-1$ dimensional hyperplane, which is a set of the form 
$$\{ \bx \in \R^n : \textbf{M} \cdot \bx = b \},$$
for some fixed $\bM \in \R^n$ and $b \in \R$.  Such a hyperplane is canonically identified with $\R^{n-1}$ so that with this identification $P_j$ is canonically identified with $Q_j \in \wp_{n-1}$, which is the meaning of $P_j \cong Q_j$.  
\end{definition} 

\begin{figure} \centering \includegraphics[width=5cm]{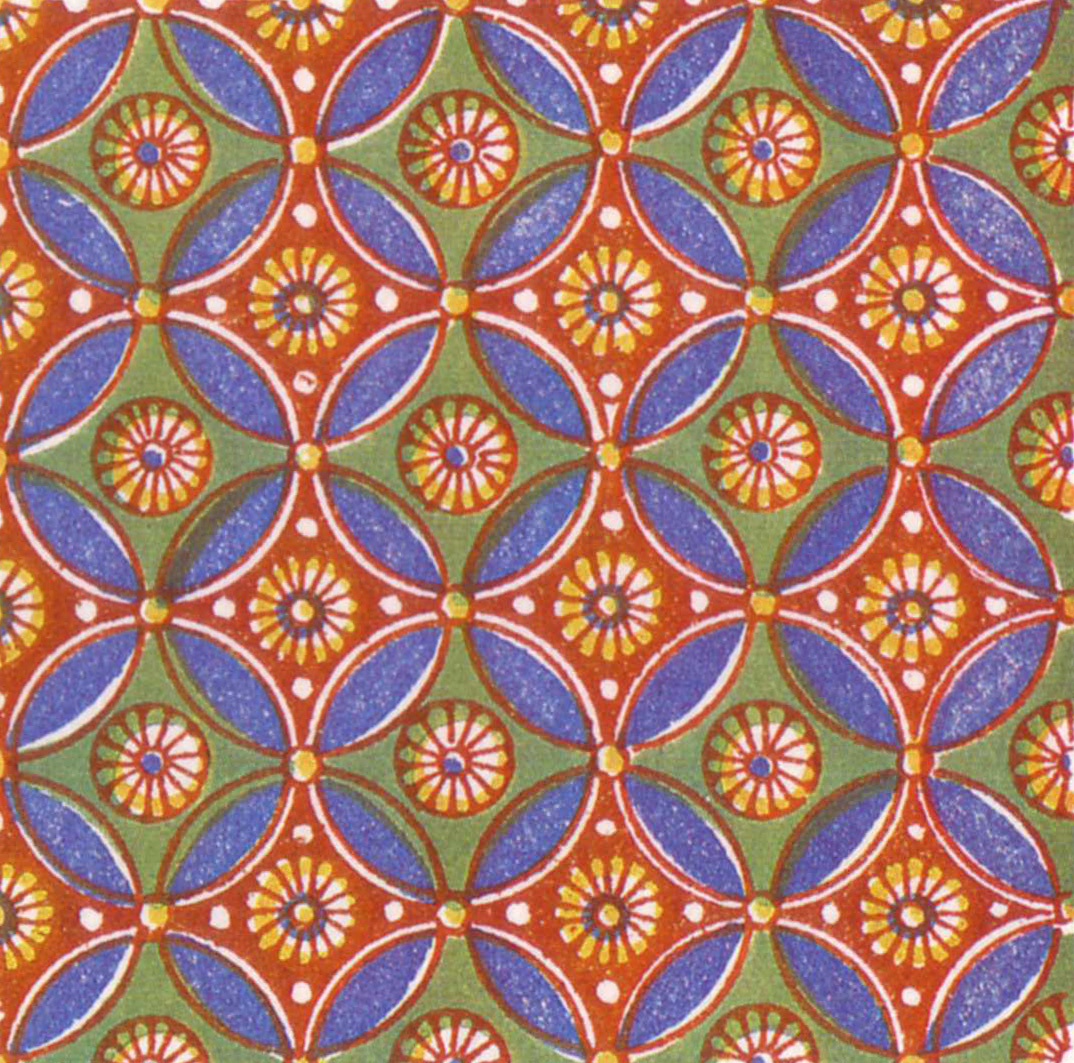} \includegraphics[width=5cm]{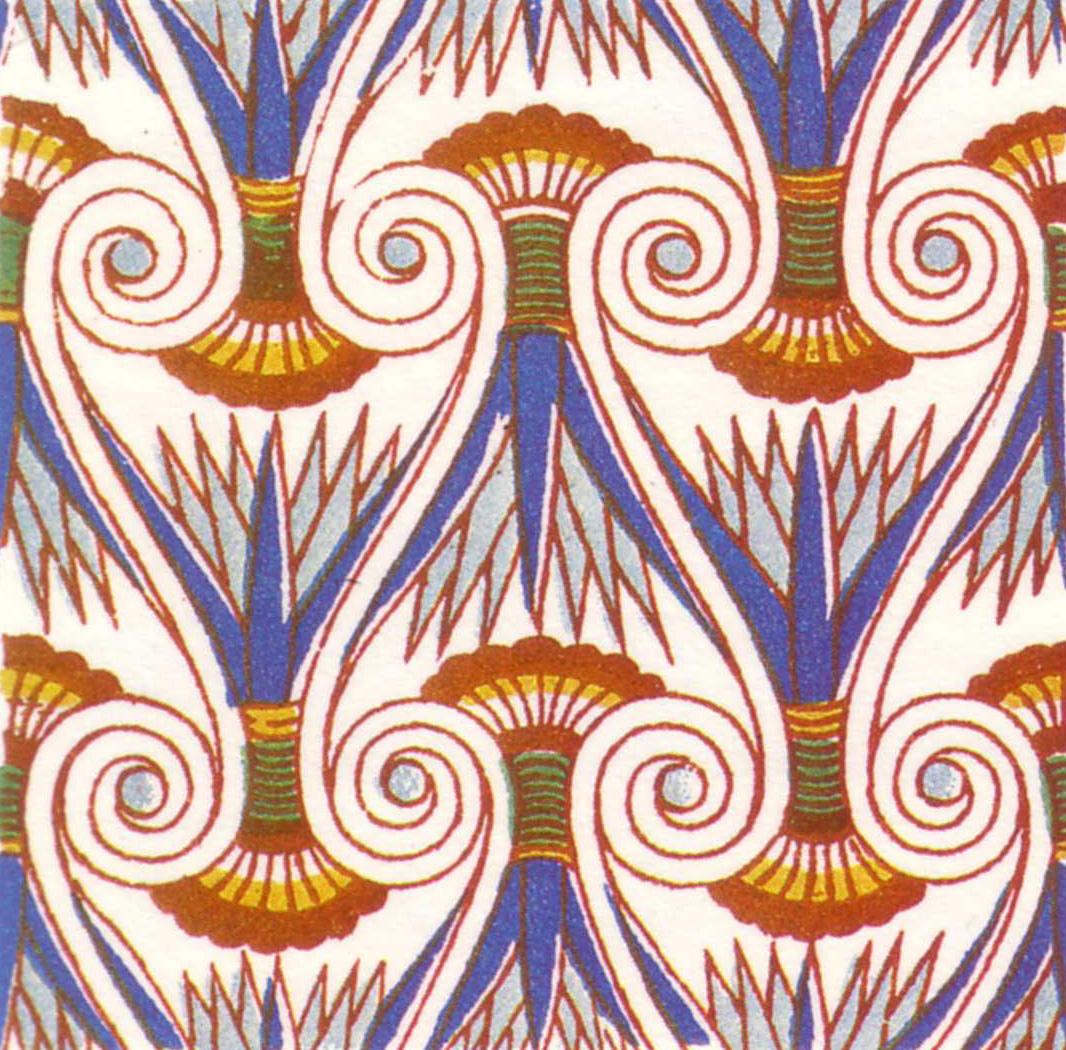} \caption{Two Egyptian patterns whose symmetry groups are planar crystallographic groups.  These patterns were documented by Owen Jones in 1856 \cite[Egyptian No 7 (plate 10), images 8 and 13]{jones}.}  \label{fig1} 
\end{figure}

Next we introduce the notion of strict tessellation. We are not aware of the term `strict tessellation' in the literature, but it may be known under a different name.  An example of a strict tessellation of the plane is given in Figure \ref{fig:tesselating}.  For an example of a tessellation of the plane which is not strict, see Figure \ref{fig:nontesselating}.  

\begin{figure}[h]
\centering
        \includegraphics[width=0.48\textwidth]{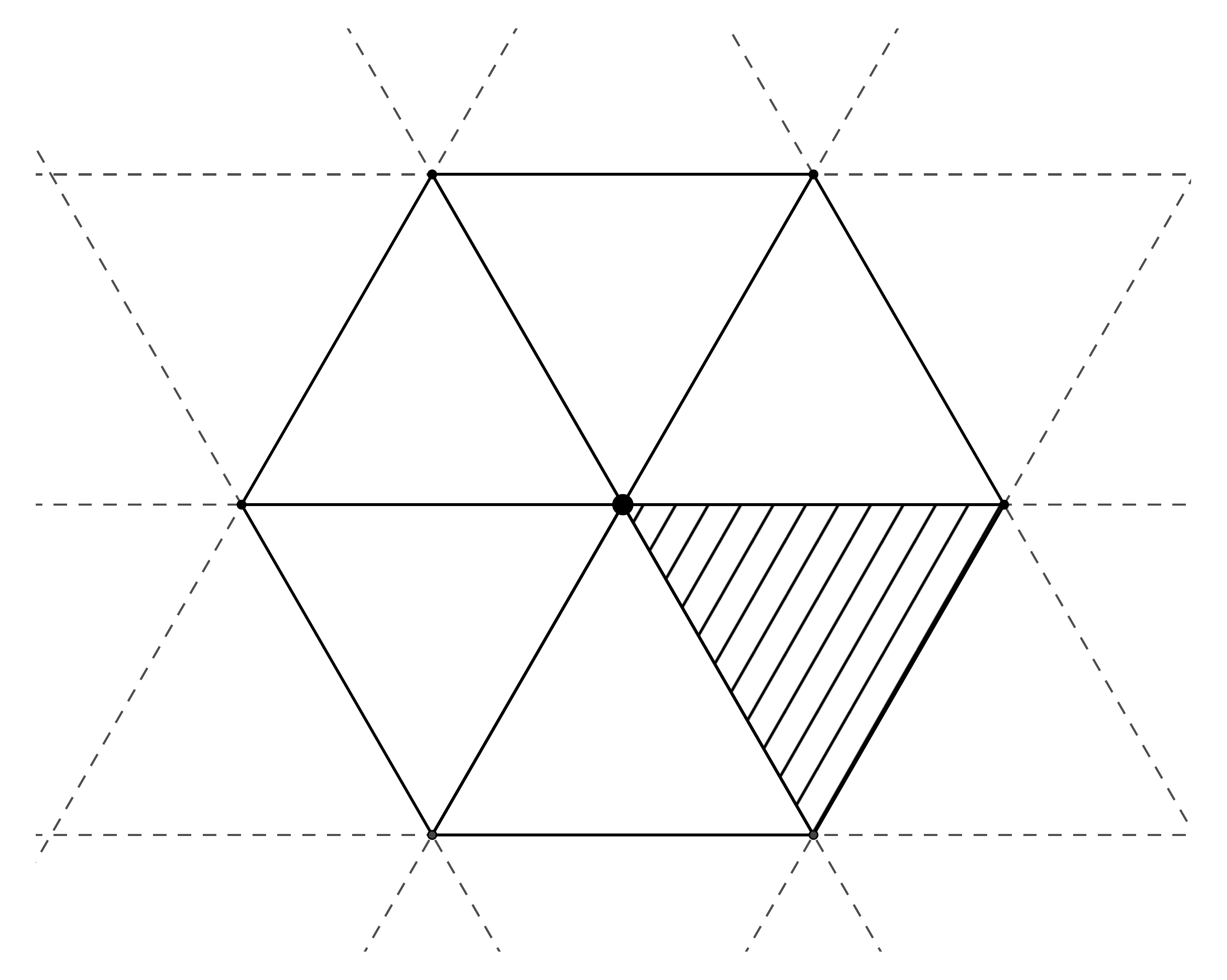}
        \caption{Equilateral triangles are shown here to strictly tessellate the plane.}
        \label{fig:tesselating}
\end{figure}

\begin{definition} \label{defst} A polytope $\Omega \in \wp_n$ strictly tessellates $\R^n$ if 
$$\R^n = \bigcup_{j\in \Z} \overline{\Omega_j}.$$
Above, each $\Omega_j$ is isometric to $\Omega$, and they are obtained by reflecting $\Omega$ across its boundary faces.  Furthermore, the hyperplanes which contain the boundary faces of each $\Omega_j$ have empty intersection with (the interior of) $\Omega_k$, for all $j$ and $k$.  
\end{definition} 

One dimensional polytopes always strictly tessellate because for any real numbers $a<b$, 
$$\R = \bigcup_{j \in \Z} \left[ j(b-a)+a, j(b-a) + b \right].$$

\begin{remark}
Any polytope covers $\R^n$ by reflection.  If there were to be a gap, that is a region of $\R^n$ not covered by reflecting the polytope across its boundary faces, such a gap will border a face of a reflected copy of the polytope. Then we simply reflect the polytope again and do this for all bordering faces until we have filled the gap.  Thus one can always cover $\R^n$ by repeatedly reflecting any polytope across its boundary faces, but there may be significant overlap.  If the (interior of the) reflected copies of the polytope have empty intersections, then we obtain a tessellation. 
\end{remark}

\begin{figure}[h]
\centering
        \includegraphics[width=0.48\textwidth]{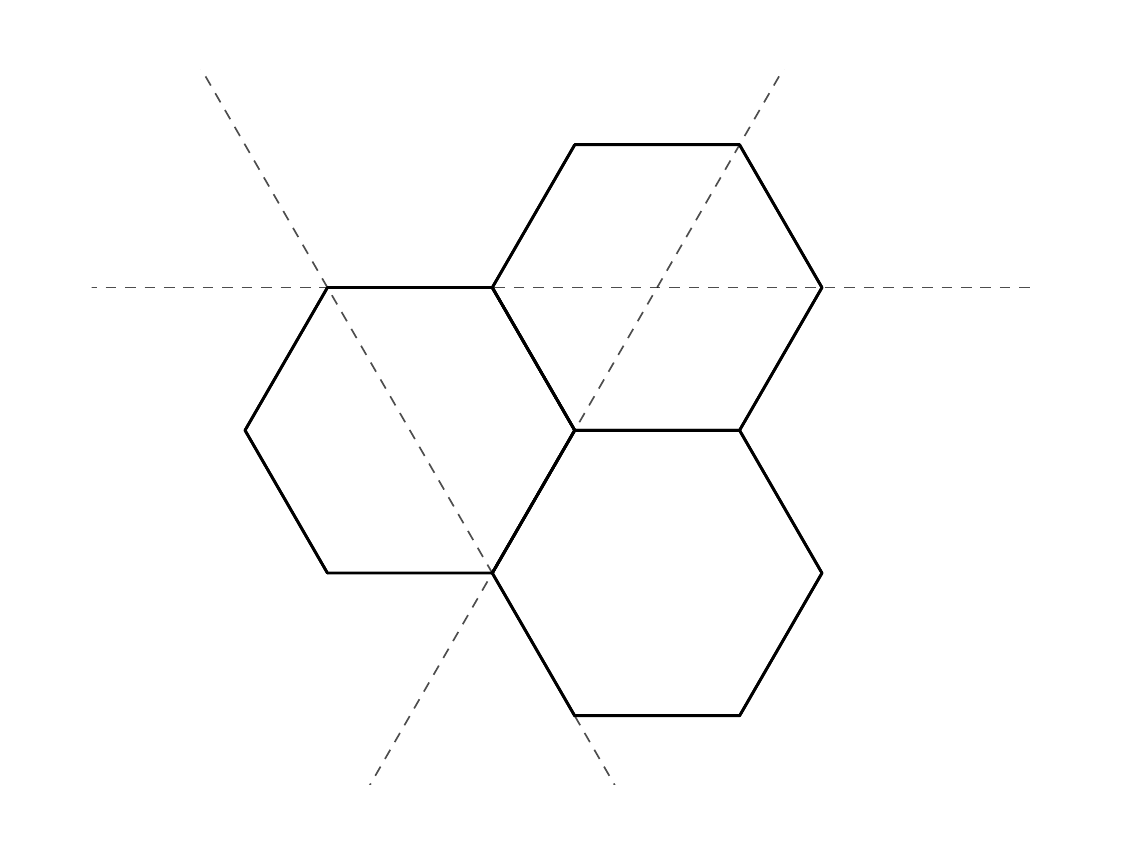}
        \caption{Although it is well known that regular hexagons tessellate the plane by reflection, the tessellation is not strict, because the lines which contain the edges of the hexagon cut through the interior of the reflected copies.}
        \label{fig:nontesselating}
\end{figure}

In 2008, McCartin proved a remarkable classification theorem \cite{mccartin}, connecting geometry and analysis.  Recall 
the Laplacian on $\R^n$ is the partial differential operator 
$$\Delta := - \sum_{k=1} ^n \frac{\pa^2}{\pa x_k ^2}.$$
The Laplace eigenvalue problem for a domain $\Omega \subset \R^n$ with the Dirichlet boundary condition is to find all functions $u$ which are not identically zero and satisfy
$$\Delta u = \lambda u, \quad \left . u\right|_{\pa \Omega} = 0, \textrm{ for some constant } \lambda.$$
This is a difficult problem, because in general it is impossible to compute these numbers $\lambda$.  However, using the tools of functional analysis \cite{courhil} one can prove that these eigenvalues are discrete and positive and therefore can be ordered 
$$0 < \lambda_1 \leq \lambda_2 \leq \ldots \uparrow \infty.$$
In this way we may speak of the first eigenfunction which has eigenvalue $\lambda_1$.  In one dimension, 
by definition \ref{def:polytope}, a polytope is a bounded open interval, $(a,b)$ for some real numbers $a<b$.  
The Laplace eigenvalue equation with the Dirichlet boundary condition on such a polytope is to find all functions $u$ such that there exists $\lambda \in \C$ with 
$$-u''(x) = \lambda u(x), \quad a < x < b, \quad u(a) = u(b) = 0.$$
Solutions to this equation are 
$$u_k (x) = \sin \left( \frac{x-a}{b-a} k \pi \right), \quad \lambda_k = \frac{k^2 \pi^2}{(b-a)^2}, \quad k \in \N.$$
Moreover, Fourier analysis \cite[Chapter 4]{folland} can be used to show that these are \em all \em the solutions to the equation.  These eigenfunctions are all trigonometric functions.  We can also define trigonometric functions on $\R^n$.  

\begin{definition} An eigenfunction $u$ for the Laplacian is \em trigonometric \em if it can be expressed as a finite sum of trigonometric functions, 
$$u(\bx) = \sum_{j=1} ^m a_j \sin(\bL_j \cdot \bx) + b_j \cos(\bM_j \cdot \bx).$$
Above, $a_j, b_j, \in \C$ and $\bL_j, \bM_j \in \R^n$ satisfy $||\bL_j||^2 = ||\bM_j||^2 = \lambda$ for all $j=1, \ldots, m$.  
\end{definition} 

\begin{remark}  \label{r:trigef} Since 
$$\cos(t) = \sin(t+\pi/2), \quad \forall t \in \R,$$
it is equivalent to define a trigonometric eigenfunction to be a function of the form 
$$u(\bx) = \sum_{j=1} ^m a_j \sin(\bL_j \cdot \bx + \phi_j).$$
Above, $a_j \in \C$, $\bL_j \in \R^n$, $\phi_j \in \R$, and $||\bL_j||$ are the same for all $j=1, \ldots, m$.  
\end{remark} 

In general, it is impossible to compute the eigenfunctions of an arbitrary polygonal domain.  Nonetheless, McCartin proved the following classification theorem which shows the equivalence of the analytic property, having trigonometric eigenfunctions, with the geometric property, strictly tessellating.  

\begin{theorem}[McCartin \cite{mccartin}] The only polygonal domains in the plane which have a complete set of trigonometric eigenfunctions for the Laplacian with the Dirichlet boundary condition are those which strictly tessellate the plane.  There are precisely four types:  rectangles, isosceles right triangles, equilateral triangles, and hemiequilateral triangles as shown in Figure \ref{fig:strikt_poly}.  
\end{theorem} 

\begin{remark} \label{r:rectangle} The Laplace eigenfunctions for a rectangular domain with vertices at the points $(0, 0)$, $(a, 0)$, $(0, b)$ and $(a, b)$ with the Dirichlet boundary condition can be computed using separation of variables which reduces the problem to two one-dimensional problems.  The resulting eigenfunctions are indexed by $m, n \in \N$.  For Cartesian coordinates $\bx = (x,y) \in \R^2$, the eigenfunctions are 
$$u_{m,n} (x,y) = \sin \left( \frac{ m \pi x}{a} \right) \sin \left( \frac{n \pi y}{b} \right).$$
Using trigonometric identities, 
$$u_{m,n} (x,y) = \frac 1 2 \left[ \cos\left( \begin{bmatrix} \frac{m \pi}{a} \\ - \frac{n \pi}{b} \end{bmatrix} \cdot \bx \right)  - \cos\left( \begin{bmatrix} \frac{m \pi}{a} \\ \frac{n \pi}{b} \end{bmatrix} \cdot \bx \right) \right].$$
Consequently, these are trigonometric eigenfunctions. 
\end{remark}

\begin{figure}[h]
\includegraphics[width=0.8\textwidth]{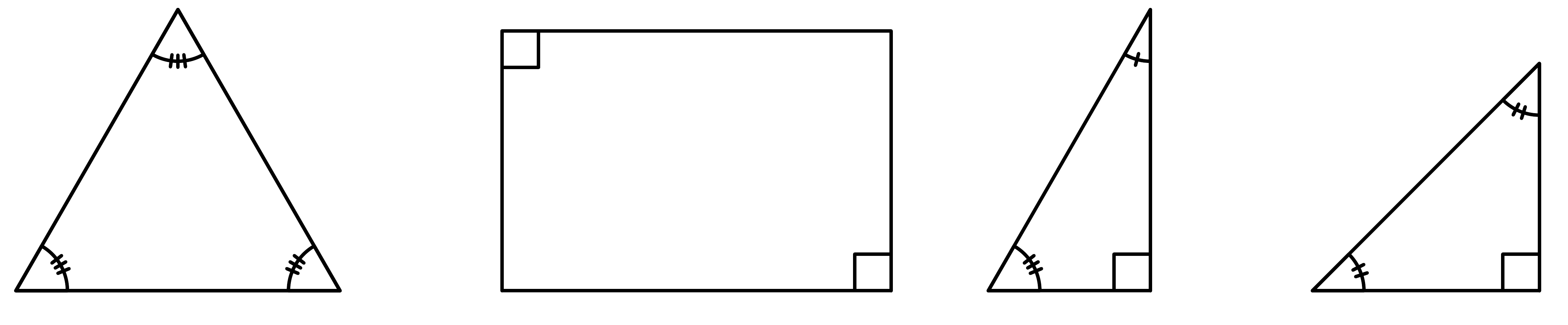} 
\caption{From left to right above we have an equilateral triangle, rectangle, hemi-equilateral triangle, and an isosceles right triangle.  These are the four types of polygons which strictly tessellate the plane.}
\label{fig:strikt_poly}
\end{figure} 

Our main result is a generalization to all dimensions.

\begin{theorem} \label{thm1} Assume that $\Omega$ is a polytope in $\R^n$. Then the following are equivalent:
\begin{enumerate} 
\item The first eigenfunction for the Laplace eigenvalue equation with the Dirichlet boundary condition extends to a real analytic function on $\R^n$. 
\item $\Omega$ strictly tessellates $\R^n$.
\item $\Omega$ is congruent to a fundamental domain of a crystallographic Coxeter group as defined in Bourbaki \cite[VI.25 Proposition 9 p. 180]{bourbaki}, and is also known as an alcove \cite[p. 179]{berard}; see also \S 3.   
\end{enumerate} 
\end{theorem} 

The three equivalent statements in Theorem \ref{thm1} are respectively analytic, geometric, and algebraic.  Our work therefore reveals an intimate connection between analysis, geometry, and algebra.  Moreover, combining our theorem with B\'erard's Proposition \cite[Proposition 9, p. 181]{berard}, we obtain the following rather remarkable result.

\begin{cor}\label{cor0} 
Assume that $\Omega$ is a polytope in $\R^n$.  If the first eigenfunction for the Laplace eigenvalue equation with the Dirichlet boundary condition extends to a real analytic function on $\R^n$, then it is a trigonometric eigenfunction.  Moreover, \em all \em the eigenfunctions of $\Omega$ are trigonometric.  
\end{cor} 

\begin{remark}  Every trigonometric eigenfunction satisfies the first condition of Theorem \ref{thm1}.  However, there are many functions which satisfy this condition but are \em not \em trigonometric.  Examples include the eigenfunctions for a disk in $\R^2$ which are products of Bessel functions and trigonometric functions.  There is no contradiction with the above corollary because a disk is not a polygonal domain.  
\end{remark}

\begin{figure}[h]
    \centering
    \scalebox{0.8}{\begin{tikzpicture}
\node[draw,circle,radius=4cm,align=center, text width=2.5cm] 
  {\small $\Omega$ is Strictly tessellating};
\node[yshift=-3cm,xshift=5cm,draw,circle,radius=4cm,align=center, text width=2.5cm] 
  {\small $\Omega$ is an Alcove};
\node[xshift=10cm,draw,circle,radius=4cm,align=center, text width=2.5cm] 
  {\small First eigenfunction analytic.};
%\coordinate (a) at (0,0)
%\coordinate (b) at (2,2)
%\draw[black,->] (2,-0.25) -- (6,-0.25);
%\node [below] at (4,-0.25) {2};
\draw[black,<->] (2,0.25) -- (8,0.25);
\node [above] at  (5,0.25)  {$\Longleftarrow$ Proven here};
%\node [above] at (5,0.8) {McCartin $\Longleftrightarrow$ in $\mathbb{R}^2$};

\draw[black,<->] (1.5,-1) -- (3.5,-2);
\node [above,rotate=-27] at  (2.2,-2.1)  {Proven here$\implies$};
%\draw[black,<-] (0.75,-1.75) -- (2.25,-3.25);
%\node [below] at (1.25,-2.5) {4};

\draw[black,<->] (6.5,-2) -- (8.5,-1);
\node [above, rotate=27] at (10-2,-2) {B\'erard [3] $\implies$};

%\draw[black,<-] (0.75+6.5,-1.75) -- (2.25+3.5,-3.25);
%\node [below] at (1.25+5.5,-2.5) {6};

\end{tikzpicture}}
    \caption{This diagram shows the three statements of Theorem \ref{thm1} and how they were proven.} 
     \label{fig:Holy_Trinity} \end{figure}
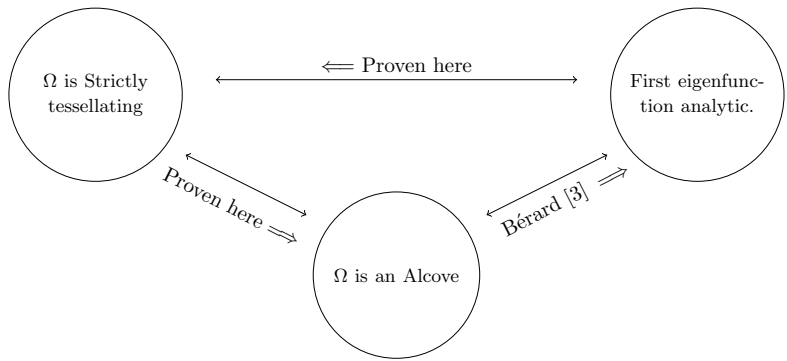

\subsection{Organization} 
In \S \ref{s:st}, we prove that if the first eigenfunction of a polytope satisfies the hypotheses of Theorem \ref{thm1}, then the polytope strictly tessellates $\R^n$.  We prove this by generalizing classical results by Lam\'e \cite{lame}.  In \S \ref{s:proof}, we introduce the notions of root systems and alcoves and prove that all polytopes which strictly tessellate $\R^n$ are alcoves.  We then recall the result of B\'erard \cite{berard}:  all alcoves have a complete set of trigonometric eigenfunctions for the Laplace eigenvalue equation with the Dirichlet boundary condition.  These results together complete the proofs of Theorem \ref{thm1} and Corollary \ref{cor0}.  
In \S \ref{s:fug} we discuss connections to the Fuglede and Goldbach conjectures.  We make our own conjecture and conclude with a purely geometric conjecture which is equivalent to the strong Goldbach conjecture.  

%\section*{Acknowledgements} JR is grateful for the support of the National Science Foundation Grant DMS-1440140 as well as a room with a view at the Mathematical Sciences Research Institute in Berkeley, California during the fall 2019 semester.  JR is grateful to Kiril Datchev, Daniel Grieser, Chris Kottke, Bob Lutz, and Rafe Mazzeo for insightful discussions which were facilitated by the NSF and MSRI.  JR is supported by the Swedish Research  Council Grant 2018-03402.  

%%%%%%%%%%%%%%%%%%%%%%%%%%%%%%%%%%%%
\section{The first eigenfunction and strict tessellation} \label{s:st} 
There is no known method to explicitly compute the eigenfunctions and eigenvalues for an arbitrary polytope.  However, using the tools of functional analysis, one can prove general facts about them.  We summarize briefly here.  For the Dirichlet boundary condition for the Laplace eigenvalue equation on a bounded domain, $\Omega \subset \R^n$, the eigenvalues form a discrete positive set which accumulates only at infinity.  We can therefore order the eigenvalues as they increase 
$$0<\lambda_1 \leq \lambda_2 \ldots \uparrow \infty.$$
We may correspondingly order the eigenfunctions.  In this way, we may speak of the ``first'' eigenfunction, which is the eigenfunction whose eigenvalue is equal to $\lambda_1$.  The eigenfunctions form an orthogonal basis of the Hilbert space $\cL^2(\Omega)$.  We shall require the following well-known fact about the first eigenfunction.  The proof of this theorem can be found in the classical PDE textbook of Evans \cite[\S 6.5]{evans}.   

\begin{theorem} \label{t:lam1} Let $\Omega$ be a bounded domain in $\R^n$.  Then the first eigenfunction for the Laplace eigenvalue equation with the Dirichlet boundary condition does not vanish in $\Omega$.   
\end{theorem}

The following result is originally due to Lam\'e \cite{lame} in two dimensions and specified to trigonometric eigenfunctions.  Here we generalize the result to $\R^n$ for all $n$ as well as to all real analytic functions.  We include a short proof for completeness.

\begin{lemma}[Vanishing planes]  Let $u$ be a real analytic function on $\R^n$.  Assume that $u$ vanishes on an open, non-empty, subset of a hyperplane, 
$$\cP := \{ \bx \in \R^n : \textbf{M} \cdot \bx = b \}.$$
Then $u$ vanishes on all of $\cP$.   
\end{lemma} 
\begin{proof}
Let $\by$ be any point of $\cP$.  Let $\bx$ be a point in the interior of the open neighborhood on which $u$ vanishes.  We parametrize a line segment to join the points $\bx$ and $\by$, defining 
$$l(t) := t\bx + (1-t) \by, \quad 0 \leq t \leq 1.$$
We note that 
$$\bM \cdot l(t) = t \bM \cdot \bx + (1-t) \bM \cdot \by \equiv b, \quad \forall t \in [0,1].$$
Consequently this line segment is contained entirely in $\cP$.  Now we consider the function 
$$u(t) := u(l(t)).$$
Since $u$ is a real analytic function, $u(t)$ is also a real analytic function of $t$.  Since $\bx$ is in the interior of the open neighborhood on which $u$ vanishes, and this neighborhood is contained in $\cP$, there is $\eps > 0$ such that $l(t)$ is contained in this neighborhood for all $t \in [1-\eps, 1]$.  The function $u(t)$ vanishes identically on $[1-\eps, 1]$.  Consequently it vanishes identically for all $t$, and in particular 
$$u(l(0)) = u(\by) = 0.$$
Since $\by \in \cP$ was arbitrary, we obtain that $u$ vanishes at every point of $\cP$.  
\end{proof} 

We will also generalize Lam\'e's Fundamental Theorem, which was originally proven in two dimensions and for trigonometric functions, to $n$ dimensions and real analytic eigenfunctions. 

\begin{theorem}[Lam\'e's Fundamental Theorem] Assume that $u$ is a real analytic function on $\R^n$ which satisfies the Laplace eigenvalue equation with the Dirichlet boundary condition on a polytope $\Omega \in \wp_n$.  Then $u$ is anti-symmetric with respect to all $(n-1)$ dimensional hyperplanes on which $u$ vanishes.
\end{theorem}

\begin{proof} 
Let $\lambda$ be the eigenvalue corresponding to $u$, so that on $\Omega$ we have 
$$\Delta u(\bx) = \lambda u(\bx) \quad \forall \bx \in \Omega.$$
Then, since $u$ is real analytic, $\Delta u$ is also real analytic on $\Omega$.  The function 
$$\Delta u - \lambda u$$
is real analytic and vanishes on $\Omega$ which is an open subset of $\R^n$.  Consequently, it vanishes on all of $\R^n$, and therefore $u$ satisfies the same Laplace eigenvalue equation on all of $\R^n$.  

Now, let $H$ be an $(n-1)$ dimensional hyperplane on which $u$ vanishes.  We consider coordinates 
$$\by = (r, \bz) \in \R^n,$$
where $\bz \in \R^{n-1}$ and 
$$r=0 \iff \by \in H.$$
Let us now define the function 
$$\widetilde u(r, \bz) := \begin{cases} u(r, \bz) & (r, \bz) \in \overline \Omega \\ - u(-r, \bz) & (r, \bz) \not \in \overline \Omega. \end{cases}$$
With this definition, $\widetilde u$ is anti-symmetric with respect to $H$.  Since $u$ satisfies the Laplace eigenvalue equation on $\R^n$ and is real analytic, the same is true for $\widetilde u$.  Moreover, on $\pa \Omega$, $u$ and $\widetilde u$ have the same Cauchy data; both functions vanish and have the same normal derivative.  Consequently, by standard uniqueness theory of partial differential equations \cite{courhil}, \cite{evans}, $u = \widetilde u$.  Therefore $u$ is also anti-symmetric with respect to $H$.   
\end{proof} 

We are now poised to prove the first implication in Theorem \ref{thm1}.  

\begin{prop} \label{prop1thm1} Assume that $\Omega$ is a polytope in $\R^n$, and the first eigenfunction satisfies the first condition of Theorem \ref{thm1}.  Then $\Omega$ strictly tessellates $\R^n$.
\end{prop} 

\begin{proof}
Let $\Omega$ be a polytope in $\R^n$ as in the statement of the proposition.  If $n=1$, then $\Omega$ is a segment and may be written as $(a,b)$ for some real numbers $a<b$.  We have computed the eigenfunctions explicitly in this case.  They are 
$$u_k (x) = \sin \left( \frac{x-a}{b-a} k \pi \right).$$
The first eigenfunction in particular satisfies the hypotheses of Theorem \ref{thm1}, and we have also shown that all one dimensional polytopes strictly tessellate $\R^1$.  Hence the proposition is proven in one dimension.  So let us assume that $n \geq 2$.  By the Vanishing Planes Lemma, for an affine subset $\cP$ which contains a boundary face of $\Omega$, all eigenfunctions of $\Omega$ vanish on $\cP$.  Since the first eigenfunction never vanishes in the interior of $\Omega$ by Theorem \ref{t:lam1}, it follows that all of the hyperplanes which contain the boundary faces of $\Omega$ have empty intersection with the interior of $\Omega$.  

Near a boundary face of $\Omega$ which is contained in the hyperplane $\cP$, we use the coordinates $(r, \bz)$ as in the previous Theorem.   Without loss of generality, we may assume that for $r>0$ sufficiently small, there are points $(r,\bz)$ inside $\Omega$.  Since $\Omega$ is a domain, it is an open set.  By Lam\'e's Fundamental Theorem, $u_1(-r, \bz) = -u_1(r,\bz)$.  Consequently, reflecting $\Omega$ across $\cP$, we see that $u_1$ is also an eigenfunction in this reflected copy of $\Omega$ which satisfies the Dirichlet boundary condition.  Since the first eigenfunction does not vanish on the interior of $\Omega$, it also does not vanish on the interior of the reflection of $\Omega$ across $\cP$.  This shows that every hyperplane which contains a boundary face of $\Omega$ has empty intersection with this reflected copy of $\Omega$.  We repeat this argument, reflecting across all of the boundary faces of $\Omega$ and then reflecting across the boundary faces of the reflected copies of $\Omega$. This shows that the hyperplanes which contain the boundary faces of the reflected copies of $\Omega$ must always have empty intersection with the interior of all reflected copies of $\Omega$.  Consequently, repeating indefinitely, we obtain a strict tessellation of $\R^n$ by $\Omega$.   
\end{proof}

%%%%%%%%%%%%%%%%%%%%%%%%%%%%%%%%%%
%\section{Classification of strictly tessellating polytopes in $\R^n$}  \label{s:proof} 
\section{Root systems, alcoves, and strictly tessellating polytopes} \label{s:proof} 
In 1980, Pierre B\'erard showed that a certain type of bounded domain in $\R^n$, known as an \em alcove, \em always has a complete set of trigonometric eigenfunctions for the Laplace eigenvalue equation with the Dirichlet boundary condition.  To define alcoves, we must first define root systems.  The concept of a root system was originally introduced by Wilhelm Killing in 1888 \cite{killing}.  His motivation was to classify all simple Lie algebras over the field of complex numbers.  In this section, we will see how our analytic problem, the study of the Laplace eigenvalue equation, is connected to these abstract algebraic concepts from Lie theory and representation theory.  

\begin{definition}\label{def:rootsystem} 
A \em root system \em in $\R^n$ is a finite set $R$ of vectors which satisfy:
\begin{enumerate} 
\item $0$ is not in $R$.
\item The vectors in $R$ span $\R^n$.
\item The only scalar multiples of $\bv \in R$ are $\pm \bv$.  
\item $R$ is closed with respect to reflection across any hyperplane whose normal is an element of $R$, that is 
$$\bv - 2 \frac{ \bu \cdot \bv}{||\bu||^2} \bu \in R, \quad \forall \bu, \bv \in R;$$
\item If $\bu, \bv \in R$ then the projection of $\bu$ onto the line through $\bv$ is an integer or half-integer multiple of $\bv$.  The mathematical formulation of this is that  
$$2 \frac{ \bu \cdot \bv}{||\bv||^2} \in \Z, \quad \forall \bu, \bv \in R.$$
\end{enumerate}
The elements of a root system are often referred to as \em roots.  \em 
\end{definition} 

\begin{remark}  There are different variations of the definition given above for a root system depending on the context.  Sometimes only conditions 1--4 are used to define a root system.  When the additional assumption (5) is included then the root system is said to be \em crystallographic.  \em  In other contexts, condition 3 is omitted, and they would call a root system which satisfies condition 3 \em reduced.  \em  
\end{remark}

We will need the dual root system to define the eigenvalues of the polytope which will be naturally associated to the root system. 

\begin{definition} \label{def:rootinverse} 
Let $R$ be a root system.  Then for $\bv \in R$ the \em coroot \em $\bv^\vee$ is defined to be 
$$\bv^\vee = \frac{2}{||\bv||^2} \bv.$$
The set of coroots $R^\vee := \{ \bv^\vee \}_{\bv \in R}$.  This is called the \em dual root system, \em and may also be called the inverse root system.  The dual root system is itself a root system.  
\end{definition}

We associate a Weyl group to a root system.  These Weyl groups are subgroups of the orthogonal group $O(n)$.

\begin{definition} \label{def:weylgroup}
For any root system $R\subset \R^n$ we associate a subgroup of the orthogonal group $O(n)$ known as its \em Weyl group. \em  This is the subgroup $W < O(n)$ generated by the set of reflections by hyperplanes whose normal vectors are elements of $R$.   For $\bv \in R$ reflection across the hyperplane with normal vector equal to $\bv$ is explicitly 
$$\sigma_\bv : \R^n \to \R^n, \quad \sigma_\bv (\bx) = \bx - 2 \frac{(\bv\cdot \bx)}{||\bv||^2} \bv.$$
\end{definition} 

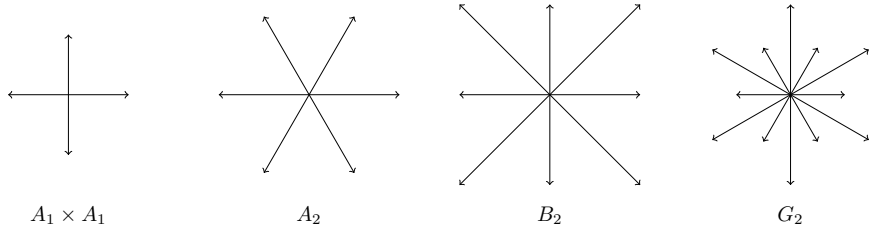
\begin{figure}[h]
    \centering
    \scalebox{0.8}{\begin{tikzpicture}
% A1 x A1
\begin{scope}
    \draw[black,<->] (-1,0) -- (1,0);
    \draw[black,<->] (0,-1) -- (0,1);
    \node at (0,-2) {$A_1 \times A_1$};
\end{scope}

\begin{scope}[shift={(4,0)}]
    \draw[black,<->] (-1.5,0) -- (1.5,0);
    \draw[black,<->] (-0.75,-1.299038) -- (0.75,1.299038);
    \draw[black,<->] (0.75,-1.299038) -- (-0.75,1.299038);
    \node at (0,-2) {$A_2$};
\end{scope}

\begin{scope}[shift={(8,0)}]
    \draw[black,<->] (-1.5,0) -- (1.5,0);
    \draw[black,<->] (0,-1.5) -- (0,1.5);
    \draw[black,<->] (-1.5,-1.5) -- (1.5,1.5);
    \draw[black,<->] (1.5,-1.5) -- (-1.5,1.5);
    \node at (0,-2) {$B_2$};
\end{scope}

\begin{scope}[shift={(12,0)}]
    \begin{scope}[rotate=-30]
        \draw[black,<->] (-1.5,0) -- (1.5,0);
        \draw[black,<->] (-0.75,-1.299038) -- (0.75,1.299038);
        \draw[black,<->] (0.75,-1.299038) -- (-0.75,1.299038);
    \end{scope}
     \begin{scope}[scale=0.6]
        \draw[black,<->] (-1.5,0) -- (1.5,0);
        \draw[black,<->] (-0.75,-1.299038) -- (0.75,1.299038);
        \draw[black,<->] (0.75,-1.299038) -- (-0.75,1.299038);
    \end{scope}
    \node at (0,-2) {$G_2$};
\end{scope}

\end{tikzpicture}}
    \caption{Here are four root systems in $\mathbb{R}^2$.  Below each root system is the name of its Weyl group. The name of the Weyl group may also be used as the name of the root system.}
    \label{fig:rotsystem}
\end{figure}

By the definition of a root system, the associated Weyl group is finite.  To explain what was proven in \cite{berard} by B\'erard, we require the notion of Weyl chamber.

\begin{definition} \label{def:chamber} 
For a root system $R \subset \R^n$ for each $\bv \in R$, let $H_\bv$ denote the hyperplane which contains the origin and whose normal vector is $\bv$.  In particular 
$$H_\bv := \{ \bx \in \R^n : \bx \cdot \bv = 0 \}.$$
Let $\mathcal H = \{ H_\bv \}_{\bv \in R}$.  Then $\R^n \setminus \mathcal H$ is disconnected, and each connected open component is known as a \em Weyl chamber. \em 
\end{definition} 

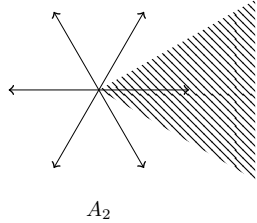
\begin{figure}[h]
    \centering
    \scalebox{0.8}{\begin{tikzpicture}
 \usetikzlibrary{patterns}
    %\fill[fill=gray!20]    (0,0) -- (0.75*2,1.299038*2) --(1.5*2,0);
    \begin{scope}[rotate=-30]
    \fill[pattern=north west lines]    (0,0) -- (0.75*2,1.299038*2) --(1.5*2,0);
    \end{scope}
     
    \draw[black,<->] (-1.5,0) -- (1.5,0);
    \draw[black,<->] (-0.75,-1.299038) -- (0.75,1.299038);
    \draw[black,<->] (0.75,-1.299038) -- (-0.75,1.299038);

    \node at (0,-2) {$A_2$};
\end{tikzpicture}}
    \caption{Extending the shaded area to infinity it is a Weyl chamber of the Weyl group $A_2$.}
    \label{fig:Weyl_chamb}
\end{figure}

To obtain an alcove, we not only wish to consider reflecting across the hyperplanes $H_v$, but also across a discrete set of parallel translations of these hyperplanes.  

\begin{figure}[h]
   \centering
    \scalebox{0.8}{\begin{tikzpicture}
\begin{scope}[scale = 1.2]
    \begin{scope}
        \draw[black,<->] (-1.5,0) -- (1.5,0);
        \draw[black,<->] (0,-1.5) -- (0,1.5);
        \draw[black,<->] (-1.5,-1.5) -- (1.5,1.5);
        \draw[black,<->] (1.5,-1.5) -- (-1.5,1.5);
        \draw[gray,dashed] (1.5,3) -- (1.5,-2);
        \draw[gray,dashed] (-3,1.5) -- (3,1.5);
        \node at (0,-3) {$B_2$};
    \end{scope}
    \begin{scope}[scale=0.707, rotate=45]
        \draw[gray,dashed] (1.5,3) -- (1.5,-3);
    \end{scope}
    \begin{scope}[scale=0.707, rotate=45]
        \draw[gray,dashed] (3,3) -- (3,-3);
    \end{scope}
    \begin{scope}[scale=0.707, rotate=-45]
        \draw[gray,dashed] (1.5,3) -- (1.5,-3);
    \end{scope}
    
    \begin{scope}[shift={(1.5,0)},scale=0.707,rotate=135]
        \filldraw[gray!20] (0,0) -- (1.5,0) -- (1.5,1.5) -- cycle;
    \end{scope}
    \node at (2,2.8) {$H_{\alpha,k}$};
    \node at (0.75,0.3) {$A$};
    \node at (1.9,0) {$\alpha$};
\end{scope}
\end{tikzpicture}}
    \caption{This shows an alcove, $A$, corresponding to the root system with Weyl group $B_2$. For $\alpha \in B_2$, the hyperplanes $H_{\alpha,k}$ for $k \in \Z$ are the parallel hyperplanes which have normal vector equal to $\alpha$. Note that $A$ is an isosceles right triangle.}
    \label{fig:alcove}
    \end{figure}
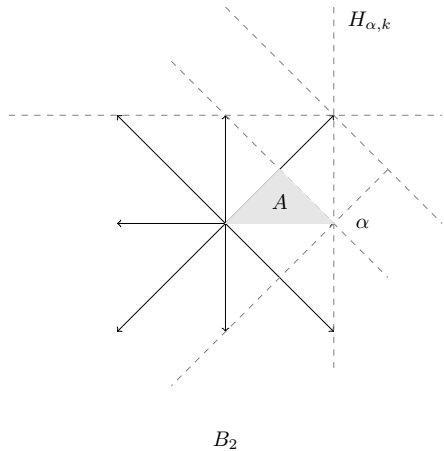

\begin{definition} \label{def:alcove} Let $R$ be a root system.  Denote by $H_\bv$ the hyperplane in $\R^n$ which contains the origin and whose normal vector is equal to $\bv$ for $\bv \in R$.  Let 
$$H_{\bv,k} = \{ \bx \in \R^n : \bv \cdot \bx = k\},$$
for $k \in \Z$.  Then $H_{\bv,0} = H_\bv$.  For $k \neq 0$, the hyperplane $H_{\bv,k}$ is parallel to $H_\bv$.  We define an \em alcove \em to be a connected component of 
$$\R^n \setminus \left\{ \bigcup_{\bv \in R, k \in \Z} H_{\bv,k} \right\}.$$
\end{definition} 

We note that the definition of an alcove immediately implies that it is a polytope in $\R^n$. An example of an alcove is shown in Figure \ref{fig:alcove}.

\begin{prop}[Proposition 9, p. 181 \cite{berard}] \label{prop1} Let $\Omega \subset \R^n$ be an alcove.  Then $\Omega$ has a complete set of trigonometric eigenfunctions for the Laplace eigenvalue equation with the Dirichlet boundary condition.  
\end{prop} 

For readers who understand French and read \cite{berard}, you may notice that the statement of Proposition \ref{prop1} is not the English translation of  \cite[Proposition 9, p. 181]{berard}.  B\'erard proved a stronger result; he specified the eigenvalues and corresponding eigenfunctions.  To understand what B\'erard proved, let $R$ be a root system.  Let $C(R)$ denote a Weyl chamber, and let $D(R)$ denote an alcove which is contained in the Weyl chamber $C(R)$.  Consider the dual root system $R^\vee$.  The vertices of the closures of the alcoves associated to $R^\vee$ create a lattice.  Let us denote this lattice by $\Gamma$.  The dual lattice 
$$\Gamma^* := \{ \bx \in \R^n : \bx \cdot \bgamma \in \Z, \quad \forall \bgamma \in \Gamma\}.$$
B\'erard referred to the points contained in this dual lattice as `the group of weights of $R$' (le groupe des poids de R) \cite{berard}.  He proved that the eigenvalues for the alcove $D(R)$ are given by 
$$\{ 4 \pi^2 ||\bq||^2 : \bq \in \Gamma^* \cap C(R)\}.$$
The multiplicity of the eigenvalue $\lambda = 4 \pi^2 ||\bq||^2$ is equal to the number of vectors $\bq \in \Gamma^* \cap C(R)$ which satisfy $\lambda = 4 \pi^2 ||\bq||^2$.  The eigenfunctions are certain linear combinations of $e^{2\pi i \bx \cdot w(\bq)}$, where $w(\bq)$ is in the affine Weyl group of $R$.  The affine Weyl group of $R$ is the semi-direct product of the Weyl group and the lattice $\Gamma$.  Combining our Proposition \ref{prop1thm1} with B\'erard's Proposition \ref{prop1} we obtain the following corollary which states that every alcove is a strictly tessellating polytope.

\begin{cor} \label{cor1} Let $\Omega \subset \R^n$ be an alcove.  Then $\Omega$ is a polytope which strictly tessellates $\R^n$.
\end{cor}

In the following proposition we prove the converse: every strictly tessellating polytope is an alcove of a root system.

\begin{prop} \label{prop2thm1} Let $\Omega \subset \R^n$ be a polytope which strictly tessellates $\R^n$.  Then $\Omega$ is an alcove.
\end{prop} 

\begin{proof} 
We will build a root system, $R$, using the fact that $\Omega$ strictly tessellates space.   The main idea is to collect all the normal vectors to all faces of $\Omega$ in the tessellation.  Since we may translate the entire tessellation, we will collect these normal directions by translating each vertex to the origin.  We will specify the lengths of the vectors later, using the fact that the tessellation is strict.  

So, begin with a vertex at the origin.  For each of the boundary faces which contains this vertex at the origin, reflection across the boundary face is an element of the orthogonal group $O(n)$.  Since the tessellation is strict, and reflection fixes the origin, after finitely many reflections across the boundary face and its images under reflection, we must return to the original copy of $\Omega$.  We begin to build the root system $R$ by defining it to contain $\pm \bv$, the normal vectors to this boundary face, and to also contain all the images of $\pm \bv$ under the repeated reflection until we return to the original copy of $\Omega$.  The lengths of the vectors will be specified later.  Since this repeated reflection returns to the original copy of $\Omega$ after finitely many reflections, the set of vectors defined in this way is finite.  Repeat this for all the boundary faces which contain the origin, and include any new normal directions $\pm \bu$ in $R$.  Next, consider a different vertex in $\Omega$.  Translate the entire tessellation so that this vertex is at the origin and repeat, including any new normal directions $\pm \bw$ in the set $R$.  Do this for all vertices.  Each time we include at most finitely many vectors.  Since the number of vertices is also finite, the total resulting set of vectors, denoted by $R$ is a finite set of vectors.  We define this set so that if $\bv \in R$ then $-\bv \in R$ as well.  This specifies the directions of all vectors, but not their lengths.  The polytope is a bounded, connected, open set with boundary consisting of flat faces.  The collection of normal vectors to the faces of $\Omega$ therefore span $\R^n$, for if not, $\Omega$ would be contained in a $k$-dimensional hyperplane in $\R^n$ and thus not an open set in $\R^n$.  Since the set $R$ contains all the normal vectors to all boundary faces of $\Omega$, the set $R$ spans $\R^n$.  As we have defined it, $0 \not \in R$.

Now fix a tessellation by $\Omega$.  Without loss of generality at least one vertex in $\overline \Omega$ is at the origin.  The tessellation defines hyperplanes in $\R^n$ which contain the boundary faces of the copies of $\Omega$ in the tessellation.  The way we have defined the set $R$, it contains all the normal directions of all these hyperplanes.  Since $R$ is a finite set, and there are countably many hyperplanes defined by the tessellation, this means that for each $\bv \in R$, we may enumerate the hyperplanes whose normal direction is $\pm \bv$ by $H_{\bv, k}$ for $k \in \Z$.  These hyperplanes are parallel.  Consider those $\bv \in R$ such that there is a hyperplane, denoted by $H_{\bv, 0}$, whose normal vector is $\pm \bv$ and which contains the origin.  We will define the length of $\bv$ so that the closest parallel hyperplanes to $H_{\bv, 0}$ are 
$$H_{\bv, \pm 1} = \{ \bx \in \R^n : \bx \cdot \bv = \pm 1\}.$$

Since the hyperplanes $H_{\bv, 1}$ and $H_{\bv,0}$ are parallel, there is a point $\by_\bv \in \R^n$ with $\by_\bv=c_\bv \bv$ for some $c_\bv \in \R$, where $\bv$ is the normal vector of (currently) unknown length.  This point $\by_\bv$ is the unique point in $H_{\bv, 1}$ which is closest to the origin.  Since $\by_\bv \in H_{\bv,1}$ we have by definition 
$$\by_\bv \cdot \bv = c_\bv ||\bv||^2 = 1.$$
We know the location of $H_{\bv, 1}$ since it is given by the tessellation, so we also know the length of $\by_\bv$.  
Using the above equation, 
$$||\by_\bv|| = |c_\bv| ||\bv||, \quad \by_\bv \cdot \bv = c_\bv ||\bv||^2 = 1 \implies c_\bv = \frac{1}{||\bv||^2} \implies ||\by_\bv|| = \frac{1}{||\bv||},$$
so we conclude that 
$$||\bv|| = \frac{1}{||\by_\bv||}.$$
Geometrically, 
$$||\by_\bv|| = \textrm{ the distance between the hyperplanes $H_{\bv,0}$ and $H_{\bv, \pm1}$}.$$
We also have 
$$||\bv|| = \frac{1}{||\by_\bv||}, \quad \by_\bv = \frac{\bv}{||\bv||^2}.$$
A schematic image of this is given in Figure \ref{fig:prop_pic}.
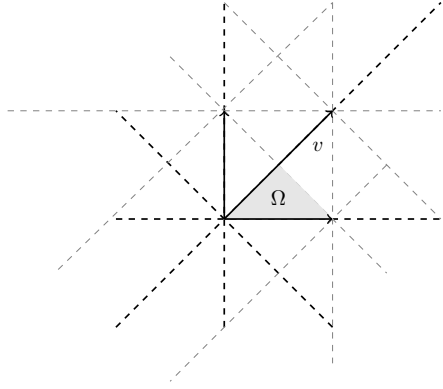
\begin{figure}[h]
   \centering
    \scalebox{0.8}{\begin{tikzpicture}
\begin{scope}[scale = 1.2]
    \begin{scope}
      \draw[black,dashed, thick] (-1.5,0) -- (3,0);
       \draw[black,dashed, thick] (0,-1.5) -- (0,3);
        \draw[black,dashed, thick] (-1.5,-1.5) -- (3,3);
       \draw[black,dashed, thick] (1.5,-1.5) -- (-1.5,1.5);
       \draw[gray,dashed] (-2.3,-0.7) -- (1.5,3);
    \draw[gray,dashed] (1.5,3) -- (1.5,-2);
        \draw[gray,dashed] (-3,1.5) -- (3,1.5);
     %   \node at (0,-3) {$B_2$};
    \end{scope}
    \begin{scope}[scale=0.707, rotate=45]
        \draw[gray,dashed] (1.5,3) -- (1.5,-3);
    \end{scope}
    \begin{scope}[scale=0.707, rotate=45]
        \draw[gray,dashed] (3,3) -- (3,-3);
    \end{scope}
    \begin{scope}[scale=0.707, rotate=-45]
        \draw[gray,dashed] (1.5,3) -- (1.5,-3);
    \end{scope}
    
    \begin{scope}[shift={(1.5,0)},scale=0.707,rotate=135]
        \filldraw[gray!20] (0,0) -- (1.5,0) -- (1.5,1.5) -- cycle;
    \end{scope}
    %\node at (2,2.8) {$H_{v,k}$};
    \node at (0.75,0.3) {$\Omega$};
    \node at (1.3,1) {$v$};
    \draw[black, ->, thick] (0,0) -- (0,1.5);
    \draw[black, ->, thick] (0,0) -- (1.5,1.5);
    \draw[black, ->, thick] (0,0) -- (1.5,0);
    
\end{scope}
\end{tikzpicture}}
    \caption{Given a polytope $\Omega$, we construct the hyperplanes $H_{\bv,0}$, here in black dotted lines, and the normal vectors $\bv$. The set $\{H_{\bv,k}\}$ includes the gray dotted lines.}
    \label{fig:prop_pic}
    \end{figure}

To define the lengths of all the elements of $R$, take each vertex in a single copy of $\Omega$, and translate the entire tessellation so that the vertex is at the origin.  Repeat the above argument to define the lengths of all the vectors in $R$.  

For $\bv \in R$, reflection across a hyperplane in the tessellation whose normal direction is that of $\bv$ may only permute the ensemble of hyperplanes by virtue of the strict tessellation.  Let $\bw \in R$.  By possibly translating the entire picture, assume that there is a hyperplane in the tessellation with normal direction $\pm \bw$ and which contains the origin, such that the origin is a vertex of a copy of $\Omega$ in the tessellation.  Thus $H_{\bw, 0}$ is a hyperplane in the tesssellation.  Consider the reflection with normal direction $\bv$, denoted by $\sigma_\bv$, that is 
$$\sigma_\bv (\bx) = \bx - 2 \frac{ \bx \cdot \bv}{||\bv||^2} \bv.$$
Then $\sigma_\bv (0) = 0$.  Consequently, $\sigma_\bv (H_{\bw, 0})$ is another hyperplane in the strict tessellation which also contains the origin, thus it is $H_{\bu, 0}$ for some $\bu \in R$.  Similarly, we also have $\sigma_\bv (H_{\bw, 1}) = H_{\bu, j}$ for some $j \in \Z$.  Since $\sigma_\bv$ preserves the scalar product, for $\bx \in H_{\bw, 1}$, by definition 
$$\bx \cdot \bw = 1 \implies \sigma_\bv (\bx) \cdot \sigma_\bv (\bw) = 1.$$
Since $\sigma_\bv$ sends $\bx$ to a point in $H_{\bu, j}$ we also have 
$$\sigma_\bv (\bx) \cdot \bu = j.$$
Since $\sigma_\bv (H_{\bw,0}) = H_{\bu,0}$, we must have that $\sigma_\bv (\bw) = \alpha \bu$ for some $\alpha \in \R$.  Therefore combining with the above we obtain
$$1 = \bx \cdot \bw = \sigma_\bv (\bx) \cdot \sigma_\bv (\bw) = \alpha \sigma_\bv (\bx) \cdot \bu = \alpha j \implies \alpha = \frac 1 j.$$
So we have proven that 
$$\sigma_\bv (\bw) = \frac 1 j \bu.$$

Now let us consider what happens to $\by_\bw$ when we reflect by $\sigma_\bv$.  The vector $\by_\bw$ is orthogonal to the hyperplanes $H_{\bw,0}$ and $H_{\bw, 1}$.  The vector goes from the origin in $H_{\bw,0}$ and since its length is the distance between $H_{\bw,0}$ and $H_{\bw, 1}$, the end of this vector sits in $H_{\bw, 1}$.  When this vector is reflected by $\sigma_\bv$, it will again start from the origin and have its endpoint lying on one of the parallel hyperplanes, by virtue of the strict tessellation.  We compute explicitly that 
$$\sigma_\bv (\by_\bw) = \by_\bw - 2 \by_\bw \cdot \bv \frac{\bv}{||\bv||^2} = \by_\bw - 2 (\by_\bw \cdot \bv) \by_\bv.$$
On the other hand, since $\sigma_\bv (\bw) = \frac 1 j \bu$, we compute using the definitions of $\by_\bw$, $\by_\bv$, and $\by_\bu$ that 
$$\sigma_\bv (\by_\bw) = \sigma_\bv \left( \frac{\bw}{||\bw||^2} \right) = \frac{1}{||\bw||^2} \sigma_\bv (\bw) = \frac{1}{||\bw||^2} \frac 1 j \bu.$$
Now, since 
$$||\bu||^2 = j^2 ||\bw||^2 \implies \sigma_\bv (\by_\bw) = j \left( \frac{\bu}{||\bw||^2 j^2} \right) = j \by_\bu.$$
Combining these calculations we obtain 
$$\sigma_\bv (\by_\bw) = \by_\bw - 2 (\by_\bw \cdot \bv) \by_\bv = j \by_\bu \implies 2 (\by_\bw \cdot \bv) \by_\bv = \by_\bw - j \by_\bu.$$
The vector $\by_\bw$ goes from the origin to $H_{\bw, 1}$, while the vector $-j \by_\bu$ goes from the origin to $H_{\bu, -j}$.  By vector addition and the strict tessellation, the sum $\by_\bw - j \by_\bu$ must go from the origin and end precisely at one of the parallel hyperplanes.  Consequently, the vector 
$$2 (\by_\bw \cdot \bv) \by_\bv$$
must be an integer multiple of $\by_\bv$ because it goes from the origin in the direction of $\by_\bv$ and lands at one of the parallel hyperplanes $H_{\bv, k}$ for some $k \in \Z$.  Therefore, 
$$2 (\by_\bw \cdot \bv) = k \in \Z.$$
By the definitions of $\by_\bw$ and $\bv$, 
$$2 (\by_\bw \cdot \bv) = 2 \frac{\bw \cdot \bv}{||\bw||^2} = k \in \Z.$$
In a similar way, reversing the roles of $\bw$ and $\bv$, we also obtain 
$$2 \frac{\bv \cdot \bw}{||\bv||^2} \in \Z.$$
Since $\bw, \bv \in R_\Omega$ were arbitrary, this shows the final condition needed for $R_\Omega$ to be a root system in Definition \ref{def:rootsystem} is satisfied.  We conclude that $R_\Omega$ is a root system and that $\Omega$ is one of its alcoves.  
\end{proof} 

The proofs of Theorem \ref{thm1} and Corollary \ref{cor0} will now follow from Propositions \ref{prop1thm1} and \ref{prop2thm1} and B\'erard's Proposition \ref{prop1}.   

\subsubsection{Proof of Theorem \ref{thm1}}
In Proposition \ref{prop1thm1} we proved that if $\Omega$ is a polytope, then the statements in Theorem \ref{thm1} satisfy:  $1 \implies 2$.  In Proposition \ref{prop2thm1}, we proved that the statements in Theorem \ref{thm1} satisfy:  $2 \implies 3$.  Finally, by B\'erard's Proposition \ref{prop1}, statement $3$, that $\Omega$ is an alcove, implies that all the eigenfunctions of $\Omega$ are trigonometric.  All trigonometric eigenfunctions are real analytic on $\R^n$ and satisfy the Laplace eigenvalue equation on $\R^n$.  Consequently, $3 \implies 1$.  The three statements are therefore equivalent.  
\qed

\subsubsection{Proof of Corollary \ref{cor0}} 
If the first eigenfunction of a polytope in $\R^n$ satisfies the hypotheses of Theorem \ref{thm1}, then the polytope is an alcove.  By B\'erard's Proposition \ref{prop1}, all of the eigenfunctions of the polytope are trigonometric.   
\qed

%%%%%%%%%%%%%%%%%%%%%%%%%%%
\section{Concluding remarks and conjectures}  \label{s:fug} 
We have now answered the analysis question:  when does a polytope in $\R^n$ have a complete set of trigonometric eigenfunctions for the Laplace eigenvalue equation?  In geometric terms, the necessary and sufficient condition for a polytope to have a complete set of trigonometric eigenfunctions is that the polytope strictly tessellates $\R^n$.  In algebraic terms, in the language of Bourbaki,  the equivalent necessary and sufficient condition is that the polytope is congruent to a fundamental domain of a crystallographic Coxeter group \cite[VI.25 Proposition 9 p. 180]{bourbaki}, \cite[p. 179]{berard}.  Returning to the analysis problem, it is interesting to note that it is enough to know that the first eigenfunction is real analytic and satisfies the Laplace eigenvalue equation on $\R^n$ to conclude that it is a trigonometric function and moreover, \em all \em the eigenfunctions are trigonometric.  This is a remarkable fact.  Moreover, the equivalence of analytic, geometric, and algebraic statements shows that these different areas of mathematics are intimately connected.  The Fuglede conjecture similarly brings together different areas of mathematics in the study of a single question.

\subsection{The Fuglede Conjecture} 
To state the Fuglede Conjecture, we introduce a few concepts.  
\begin{definition} 
A domain $\Omega \subset \R^d$ is said to be a \em spectral set \em if there exists $\Lambda \subset \R^n$ such that the functions 
$$\{ e^{2\pi i \lambda \cdot \bx} \}_{\lambda \in \Lambda}$$
are an orthogonal basis for $\cL^2 (\Omega)$.  The set $\Lambda$ is then said to be \em a spectrum of $\Omega$, \em  and $(\Omega, \Lambda)$ is called \em a spectral pair.  \em 
\end{definition} 

To relate these notions to our work here, we observe that if a domain $\Omega$ were to have all its eigenfunctions for the Laplace eigenvalue equation of the form $e^{2\pi i \lambda \cdot \bx}$, then these functions would comprise an orthogonal basis for $\cL^2 (\Omega)$. Consequently, knowing that the eigenfunctions are precisely of this form implies that the domain is a spectral set.  However, the converse is not true, in the sense that if $\Omega$ is a spectral set, then its eigenfunctions are not necessarily individual complex exponential functions.  If $\Omega$ is a spectral set, then the eigenfunctions must be linear combinations of the $e^{2\pi i \lambda \cdot \bx}$, since these are a basis for $\cL^2 (\Omega)$.  However, the linear combinations could have countably infinitely many terms, so it is not clear what precise form the eigenfunctions will take.  

\begin{conj}[Fuglede \cite{fug}] Every domain of $\R^n$ which has positive Lebesgue measure is a spectral set if and only if it tiles $\R^n$ by translation.  
\end{conj} 

Fuglede proved in 1974 that the conjecture holds if one assumes that the domain is the fundamental domain of a lattice \cite{fug}.  Only several years later, in 2003, was further progress made by Iosevich, Katz, and Tao \cite{ios} who proved that the Fuglede conjecture is true if one restricts to convex planar domains.  In the following year, Tao proved that the Fuglede Conjecture is false in dimension 5 and higher \cite{tao}.  In 2006, the works of Farkas, Kolounzakis, Matolcsi and Mora \cite{fkmm}, \cite{km}, \cite{km2}, \cite{mat} proved that the conjecture is also false for dimensions 3 and 4.  In 2017, Greenfeld and Lev proved that Fuglede's Conjecture is true if one restricts attention to domains which are convex polytopes, but only in $\R^3$ \cite{green}.  In 2019, Lev and Matolcsi proved that Fuglede's Conjecture is true if one restricts attention to convex domains, in any dimension \cite{lm}.  Interestingly, the Fuglede Conjecture is still an open problem for arbitrary domains in dimensions one and two.  Here we make the following conjecture which is related to yet independent from Fuglede's.  

\begin{conj} Let $\Omega$ be a domain in $\R^n$.  Then $\Omega$ has a complete set of trigonometric eigenfunctions for the Laplace eigenvalue equation with the Dirichlet boundary condition if and only if $\Omega$ is a polytope which strictly tessellates $\R^n$.  Equivalently, $\Omega$ has a complete set of trigonometric eigenfunctions for the Laplace eigenvalue equation with the Dirichlet boundary condition if and only if $\Omega$ is an alcove. 
\end{conj} 

\begin{figure}[h] \centering
\includegraphics[width=8cm]{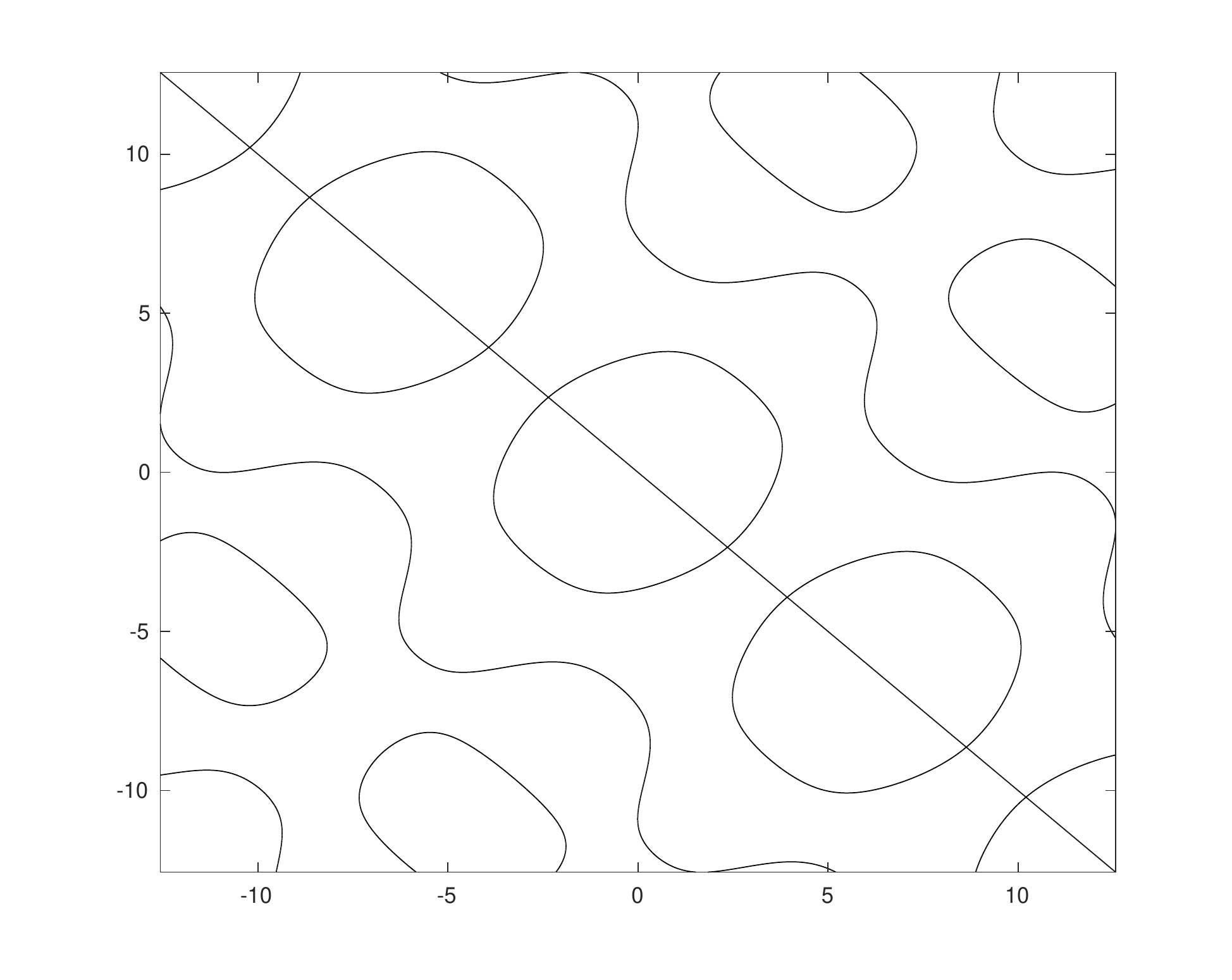} 
\caption{This figure shows the null set of the function 
$u(x,y) = \sin(x)+\sin(y)+\sin((x+y)/\sqrt{2})$ in a square shaped region of $\R^2$.  The null set includes the line $y=-x$ as well as the other curves in the region.  Consequently,  by uniqueness, this function is the first eigenfunction of the connected, open domains which are bounded by these curves, since it vanishes on the boundary but not on the interior and satisfies the Laplace eigenvalue equation.  Hence the first eigenfunction satisfies the first condition of Theorem \ref{thm1}, but we do not obtain any further conclusions because the domain is not a polytope.}
\label{fig:nullset} 
\end{figure}

The difficulty in treating arbitrary domains is that we do not have a replacement for Lam\'e's results which are central to our proof.  Moreover, it is possible to construct linear combinations of trigonometric functions which vanish on curved regions; an example is given in Figure \ref{fig:nullset}. Consequently, we cannot immediately conclude that domains which have trigonometric eigenfunctions have flat boundary faces, and hence they are polytopes.  A domain with a curved boundary could have a few trigonometric eigenfunctions.  What is reasonable to expect, however, is that it does not have a \em complete \em set of trigonometric eigenfunctions.  
%%%%%%%%%%%%%%%%%%%%%%%%%%%%%%%%%%%%%
\subsection{The crystallographic restriction theorem and a geometric approach to the Goldbach Conjecture} 
The vertices of the strict tessellation given by a polytope which is an alcove are in fact the set of points in a full rank lattice.  We note that two different polytopes may give rise to the same lattice; for example an isosceles right triangle and the square obtained by two copies of that triangle will produce the same lattice.   For any discrete group of isometries of $\R^n$, an element $g$ of such a group has finite order if there is an integer $k>0$ such that $g$ composed with itself $k$ times is the identity.  The minimal such $k$ is the order of $g$.  To state the crystallographic restriction theorem, we define a function which is like an extension of the Euler totient function.  For an odd prime $p$ and $r\geq 1$, 
$$\psi(p^r) := \phi(p^r), \quad \phi(p^r) = p^r - p^{r-1}.$$  
Above, $\phi$ is the Euler totient function.  The Euler totient function of a positive integer $n$ counts the positive integers which are relatively prime with and less than or equal to $n$.  So, for example, for an odd prime $p$, the positive integers which are not relatively prime with $p^r$ are $p$, $2p$, $3p$, \ldots $p^{r-1} p = p^r$.  There are $p^{r-1}$ of these.  All other positive integers are relatively prime with $p^r$, hence $\phi(p^r) = p^r - p^{r-1}$.  The function $\psi$ is further defined as follows:
$$\psi(1) = \psi(2) = 0, \,  \psi(2^r) :=\phi(2^r) \textrm{ for } r>1,$$ 
and 
$$\textrm{ for } m = \prod_i p_i ^{r_i}, \quad  \psi(m) := \sum_i \psi(p_i ^{r_i}).$$

\begin{theorem}[Crystallographic Restriction I]  For any discrete group $G$ of isometries of $\R^n$, for $n \geq 2$, the set of orders of the elements $G$ which have finite order is equal to 
$$\Ord_n = \{ m \in \N : \psi(m) \leq n \}.$$
\end{theorem} 

The crystallographic restriction theorem is connected to the mathematics of crystals when we reformulate the theorem in the context of lattices.  A full-rank lattice is a set of points in $\R^n$ of the form
\begin{equation} \label{lattice1} \Gamma = \{ \bp \in \R^n : \bp = L \bx, \quad L \in \GL(n, \R), \quad \bx \in \Z^n \}. \end{equation} 
Above, $\GL(n, \R)$ is the set of $n \times n$ invertible matrices with real entries, and $\Z^n$ are the elements of $\R^n$ whose entries are integers.  We say that the matrix $L$ generates the lattice $\Gamma$.  The generating matrix $L$ is not unique, because for any $M \in \GL (n , \Z)$ the set of points in \eqref{lattice1} is equal to 
$$ \{ \bp \in \R^n : \bp = LM \bx, \quad \bx \in \Z^n \}.$$ 
Here $\GL(n, \Z)$ is the group of invertible $n \times n$ matrices whose entries are integers. Note that to be a group, this requires the determinant of all elements of $\GL(n, \Z)$ to be equal to $\pm 1$.  Two matrices $L_1, L_2 \in \GL(n, \R)$ generate the same lattice if and only if there is an $M \in \GL(n, \Z)$ such that $L_1 = L_2 M$.  For a matrix $M \in \GL(n, \Z)$, we identify it with the isometry of $\R^n$ which maps $\bx \in \R^n$ to $M\bx$. The matrices in $\GL(n, \Z)$ can therefore be identified with the group of symmetries of the crystal whose atoms lie on the points of the lattice.   Hence, the order of $M$ is equal to the smallest positive integer $k$ such that $M^k$ is the identity matrix.  It turns out that the set of orders of the elements of any discrete group $G$ of isometries of $\R^n$ which have finite order is equal to the set of orders of the elements of $\GL(n, \Z)$.  Consequently, the crystallographic restriction may be reformulated as follows.  

\begin{theorem}[Crystallographic Restriction II] \label{th:crystal} For any $n \geq 2$, the set of orders of the elements of $\GL(n, \Z)$ is equal to 
$$\Ord_n = \{ m \in \N : \psi(m) \leq n \}.$$
\end{theorem} 

In \cite{goldbach}, Bamberg, Cairns and Kilminster proved that one may reformulate the Strong Goldbach Conjecture in terms of the orders of elements of $\GL (n, \Z)$.  

\begin{conj}[Strong Goldbach] Every even natural number greater than six can be written as the sum of two distinct odd primes.  
\end{conj} 

\begin{theorem}[Theorem 3 of \cite{goldbach}] \label{th:gold} The following statements are equivalent:
\begin{enumerate} 
\item The strong Goldbach conjecture is true; 
\item For each even $n \geq 6$ there is a matrix $M \in \GL(n, \Z)$ which has order $pq$ for distinct primes $p$ and $q$, and there is no matrix in $\GL(k, \Z)$ of order $pq$ for any $k < n$.
\end{enumerate} 
\end{theorem}

The Goldbach Conjecture is an extremely difficult problem.  Difficult, long-standing open problems have sometimes been solved by translating the problem into a different field of mathematics.  The proof of Fermat's Last Theorem, also a statement in number theory, was achieved using newly developed techniques in algebraic geometry \cite{wiles1}, \cite{wiles2}.  To approach the Goldbach Conjecture geometrically, we ask 
\begin{question} \em Is there a geometric reason for the existence of a symmetry for full rank lattices in $\R^n$, with $n\geq 6$ an even number, such that this symmetry is of order $pq$ for two odd primes $p \neq q$ such that $p+q=n+2$?
\em 
\end{question} 
The condition that there is no matrix in $\GL(k, \Z)$ of order $pq$ for any $k<n$ is equivalent to requiring $p+q = n+2$.  This follows from the Crystallographic Restriction Theorem \ref{th:crystal} which states that the orders of the elements of $\GL(k, \Z)$ is equal to the set of non-negative integers $m$ with $\psi(m) \leq k$.  In order to guarantee that 
$$\psi(pq) = p+q-2 > k \textrm{ for all } k < n, \textrm{ but } \psi(pq) \leq n \implies \psi(pq) = p+q-2 = n.$$

Consequently, the symmetry of order $pq$ would correspond to a matrix $M \in \GL(n, \Z)$ which does not admit a diagonal decomposition into two matrices of smaller dimensions.  Geometrically, this matrix would not arise as a product of symmetries of $\R^k$ and $\R^{n-k}$ for any $k=1, \ldots, n-1$.  It would be a new symmetry occurring first in $\R^n$.  Since \cite{goldbach} already realized the connection between the Goldbach Conjecture and the Crystallographic Restriction Theorem, this geometric approach is pretty unlikely to lead to any new developments.  Nonetheless, it is interesting that a famous number theoretic conjecture can be equivalently phrased as a simple question about the orders of symmetries of full-rank lattices in $\R^n$.


\begin{thebibliography}{99} 

\bibitem{graph2} N. A. Abdel Ghany, S. A. Elsherif and H. T. Handal, \em Revolution of Graphene for different applications:  State-of-the-art, \em Surfaces and Interfaces, vol. 9, 93--106 (2017). 

\bibitem{goldbach} J. Bamberg, G. Cairns, and D. Kilminster, \em The Crystallographic Restriction, Permutations, and Goldbach's Conjecture, \em The American Math Monthly, vol. 110, no. 3, 202--209, (2003).  

\bibitem{berard} P. B\'erard, \em Spectres et groupes cristallographiques I:  Domaines euclidiens, \em Inventiones Math, 58, 179--199 (1980).  

\bibitem{bieber1} L. Bieberbach, \em Ueber die Bewegungsgruppen der Euklidischen Räume I, \em  Math. Ann. 70,  297--336, (1911).  

\bibitem{bieber2} L. Bieberbach, \em Ueber die Bewegungsgruppen der Euklidischen Räume II, \em  Math. Ann.  72, 400--412 (1912).  

\bibitem{bourbaki} N. Bourbaki, \em Groupes et alg\`ebres de Lie, Chapitres 4 \`a 6. \em  Act. Scient. et Industrielles 1337 Paris:  Hermann (1968).  

\bibitem{courhil} R. Courant, D. Hilbert, \em  Methoden der Mathematischen Physik, \em Vols. I, II,
Interscience Publishers, Inc., New York, (1943).  

\bibitem{evans} L. C. Evans, \em Partial Differential Equations, \em Second Edition, Series: Graduate Studies in Mathematics (Book 19), American Mathematical Society, (2010) ISBN-10: 0821849743, ISBN-13: 978-0821849743. 

\bibitem{fkmm} B. Farkas, M. Kolounzakis, M. Matolcsi, P. M\'ora, \em On Fuglede's conjecture and the existence of universal spectra, \em  J. Fourier Anal. Appl. 12 (5), (2006), 483--494. 

\bibitem{fed} E. S. Fedorov, \em Symmetry in the plane, \em Proc. of the Imperial St. Petersburg Mineralogical Society, series 2, 28, 345--390, (1891).  (Russian) 

\bibitem{folland} G. B. Folland, \em Fourier analysis and its applications, \em The Wadsworth \& Brooks/Cole Mathematics Series. Wadsworth \& Brooks/Cole Advanced Books \& Software, Pacific Grove, CA, (1992).


\bibitem{fug} B. Fuglede, \em Commuting self-adjoint partial differential operators and a group theoretic problem, \em  J. Func. Anal. 16, (1974), 101--121. 

\bibitem{green} R. Greenfeld and N. Lev, \em Fuglede's spectral set conjecture for convex polytopes, \em to appear in Analysis \& PDE, \url{https://arxiv.org/abs/1602.08854}.

\bibitem{ios} A. Iosevich, N. Katz, and T. Tao, \em The Fuglede spectral conjecture hold for convex planar domains, \em Math. Res. Lett. 10 (5–6), (2003), 556--569.

\bibitem{jones} O. Jones, \em The grammar of ornament, \em London:  Day and Son, (1856).  

\bibitem{killing} W. Killing, \em Die Zusammensetzung der stetigen endlichen Transformationsgruppen, \em  Mathematische Annalen, part 1: vol.  31, no. 2, (1888), 252--290, part 2, vol. 33, no. 1 (1888), 1--48, part 3, vol. 34, no. 1 (1889), 57--122, vol. 36, no. 2 (1890), 161--189.  

 \bibitem{km} M. Kolounzakis and M. Matolcsi, \em Tiles with no spectra, \em Forum Math. 18 (3), (2006), 519--528.
 
 \bibitem{km2} M. Kolounzakis and M. Matolcsi, \em  Complex Hadamard Matrices and the spectral set conjecture, \em  Collect. Math. Extra, (2006), 281--291.

\bibitem{lame} G. Lam\'e, \em M\'emoire sur la propagation de la chaleur dans les poly\`edres, \em Journal de l'\'Ecole Polytechnique, vol. 22, (1833), 194--251.  

\bibitem{lm} N. Lev and M. Matolcsi, \em The Fuglede conjecture for convex domains is true in all dimensions, \em \url{https://arxiv.org/abs/1904.12262}. 

\bibitem{mat} M. Matolcsi, \em Fuglede's conjecture fails in dimension 4, \em  Proc. Amer. Math. Soc. 133 (10), (2005), 3021--3026. 

\bibitem{mccartin} B. J. McCartin, \em On polygonal domains with trigonometric eigenfunctions of the Laplacian under Dirichlet or Neumann boundary conditions, \em Applied Math Sciences, vol. 2, no. 58, 2891--2901, (2008).  

\bibitem{graph1} K. S. Novoselev, A. K. Geim, S. V. Morozov, D. Jiang, Y. Zhang, S. V. Dubonos, I. V. Grigorieva, and A. A> Firsov, \em Electric Field Effect in Atomically Thin Carbon Films, \em Science, vol. 306, issue 5696, 666--669, (2004).  

\bibitem{graph3} K. S. Novoselov, V. I. Fal'ko, L. Colombo, P. R. Gellert, M. G. Schwab, and K. Kim, \em A roadmap for graphene, \em Nature, 490, 192--200 (2012). 

\bibitem{graph4} D. G. Papageorgiou, I. A. Kinloch, and R. J. Young, \em Mechanical properties of graphene and graphene-based nanocomposites, \em Progress in Materials Science vol. 90, 75--127, (2017).

\bibitem{schoen} A. Schoenflies, \em Kristallsysteme und Kristallstruktur, \em Teubner, (1891).  

\bibitem{tao} T. Tao, \em Fuglede's conjecture is false on 5 or higher dimensions, \em Math. Res. Lett. 11 (2–3), (2004), 251--258.

\bibitem{wiles1} A. Wiles, \em Modular elliptic curves and Fermat's Last Theorem, \em Annals of Mathematics. 141 (3), (1995), 443--551. 

\bibitem{wiles2} A. Wiles \em Ring theoretic properties of certain Hecke algebras, \em Annals of Mathematics. 141 (3), (1995) 553--572. 

\bibitem{graph5} Y. Zheng, Z. Zhen, and H. Zhu, \em Graphene:  Fundamental research and potential applications, \em FlatChem vol. 4, 20--32, (2017).  


\end{thebibliography}
\end{document}